\documentclass[11pt, reqno]{amsart}
\usepackage[utf8]{inputenc}
\usepackage{indentfirst, amssymb, amsmath, amsthm, mathrsfs, setspace, indentfirst, enumerate,  mathrsfs, amsmath, amsthm}
\usepackage[colorlinks=true,linkcolor=purple, citecolor=blue,urlcolor=magenta]{hyperref}
\usepackage{graphicx}      
\usepackage{float}         
\usepackage{xcolor}        
\usepackage{colortbl}     
\usepackage{multirow}     
\usepackage{tikz}          
\definecolor{sectionlink}{RGB}{0,100,200} 
\newtheorem{ques}{Question}[section]

\newtheorem{theo}{Theorem}[section]
\newtheorem{lem}{Lemma}[section]

\newtheorem{exm}{Example}[section]
\newtheorem{defi}{Definition}[section]
\newtheorem{rem}{Remark}[section]
\newtheorem*{theoA}{Theorem A}
\newtheorem*{theoB}{Theorem B}
\newtheorem*{theoC}{Theorem C}
\newtheorem*{theoD}{Theorem D}
\newtheorem*{theoE}{Theorem E}

\numberwithin{equation}{section}
\newcommand{\beas}{\begin{eqnarray*}}
\newcommand{\eeas}{\end{eqnarray*}}
\newcommand{\bea}{\begin{eqnarray}}
\newcommand{\eea}{\end{eqnarray}}

\begin{document}

\title[ Hankel, Toeplitz, Hermitian--Toeplitz Determinants and the Zalcman Functional...]{Hankel, Toeplitz, Hermitian--Toeplitz Determinants and the Zalcman Functional for a Class of Biholomorphic Mappings on Complex Banach Spaces}
\author[ N. Sarkar and P. Das]{Nabadwip Sarkar and Pradip Das}
\address{Amity School of Applied Sciences, Amity University Mumbai, Panvel, Navi Mumbai, Maharashtra-410206, India}
\email{nsarkar@mum.amity.edu, nabadwipsarkar52@gmail.com}
\address{Department of Mathematics, Raiganj University, Raiganj, West Bengal-733134, India.}
\email{pradipsmath@gmail.com}

\makeatletter
\@namedef{subjclassname@2020}{\textup{2020} Mathematics Subject Classification}
\makeatother


\begin{abstract} In this paper, we investigate the second Hankel determinant, Toeplitz determinants, Hermitian--Toeplitz determinants and the generalized Zalcman functional for a class of normalized biholomorphic mappings on complex Banach spaces. Motivated by recent results for the corresponding class of analytic functions in the unit disk, we establish Banach space analogues of these determinant estimates. Our results extend the one-variable inequalities of Allu, Lecko and Thomas \cite{ALT2022} to the setting of complex Banach spaces, thereby providing higher-dimensional counterparts for the class corresponding to $\mathcal{F}_{O}(\lambda)$.

\end{abstract}

\maketitle

\section{Introduction and Preliminaries}

Let $\mathcal{A}$ denote the family of analytic functions in the open unit disk $\mathbb{U}=\{z\in\mathbb{C}:|z|<1\},$
normalized by
\begin{equation}\label{eq:A}
f(z)=z+\sum_{m=2}^{\infty}a_mz^m.
\end{equation}
Furthermore, let $\mathcal{S}\subset\mathcal{A}$ be the class of normalized univalent functions, and let $\mathcal{K}$ denote the subclass of $\mathcal{S}$ consisting of convex functions.

Throughout this paper, we shall frequently use the classical Carath\'eodory class, denoted by $\mathcal{P}$, which consists of analytic functions $p$ satisfying
\[
p(0)=1
\quad\text{and}\quad
\operatorname{Re}p(z)>0,\qquad z\in\mathbb{U}.
\]
Every function $p\in\mathcal{P}$ admits the Taylor expansion
\begin{equation}\label{eq:P}
p(z)=1+\sum_{m=1}^{\infty}p_mz^m
=1+p_1z+p_2z^2+p_3z^3+\cdots,
\qquad z\in\mathbb{U}.
\end{equation}

Let $X$ be a complex Banach space endowed with the norm $\|\cdot\|$, and let
\[
\mathbb{B}=\{x\in X:\|x\|<1\}
\]
be its open unit ball. We denote by $\mathcal{L}(X,Y)$ the Banach space of all bounded linear operators from $X$ into another complex Banach space $Y$. The identity operator on $X$ is denoted by $I$.

For each nonzero vector $x\in X$, define
\[
T(x)=
\left\{
T_x\in\mathcal{L}(X,\mathbb{C}):
T_x(x)=\|x\|,
\ \|T_x\|=1
\right\}.
\]
The Hahn--Banach theorem guarantees that $T(x)$ is nonempty. Moreover, for every $\xi\in\mathbb{C}\setminus\{0\}$,
\[
T_{\xi x}(\cdot)=\frac{|\xi|}{\xi}\,T_x(\cdot),
\]
which establishes a one-to-one correspondence between $T(\xi x)$ and $T(x)$.

Let $\mathcal{H}(\mathbb{B})$ denote the family of all holomorphic mappings from $\mathbb{B}$ into $X$. If $F\in\mathcal{H}(\mathbb{B})$, then for every $x\in \mathbb{B}$, the Fr\'echet--Taylor expansion is given by
\[
F(y)
=
\sum_{n=0}^{\infty}
\frac{1}{n!}
D^nF(x)\bigl((y-x)^n\bigr),
\]
for all $y$ sufficiently close to $x$, where
\[
D^nF(x)\bigl((y-x)^n\bigr)
=
D^nF(x)(y-x,\ldots,y-x),
\]
and $D^nF(x)$ denotes the $n$th Fr\'echet derivative of $F$ at $x$. Each $D^nF(x)$ is a bounded symmetric $n$-linear operator from $X^n$ into $X$.

A mapping $F\in\mathcal{H}(\mathbb{B})$ is called \emph{biholomorphic} if $F(\mathbb{B})$ is a domain in $X$ and the inverse mapping
\[
F^{-1}:F(\mathbb{B})\rightarrow \mathbb{B}
\]
exists and is holomorphic. Moreover, $F$ is said to be \emph{locally biholomorphic} whenever the Fr\'echet derivative $DF(x)$ is invertible with bounded inverse for every $x\in \mathbb{B}$.

A holomorphic mapping $F:B\rightarrow X$ is said to be \emph{normalized} if
\[
F(0)=0
\quad\text{and}\quad
DF(0)=I.
\]
Several important subclasses of normalized biholomorphic mappings in Banach spaces can be found in the monograph~\cite{GK2003}.

For later use, let $x_0\in X$ satisfy $\|x_0\|=1$ and choose $T_{x_0}\in T(x_0)$. We introduce the quantities $A_1=1$
and, for $n=2,3,4$,
\begin{equation}\label{eq:An}
A_n=\frac{1}{n!}T_{x_0}\left(D^nF(0)(x_0,\ldots,x_0)
\right),
\end{equation}
where the vector $x_0$ appears exactly $n$ times in the multilinear form $D^nF(0)$.

We shall also require the following definitions.

\begin{defi}
Let $f \in \mathcal{A}$ be locally univalent in the unit disk $\mathbb{D}$, and let
$-\frac{1}{2}<\lambda\leq 1$. Then $f\in \mathcal{F}(\lambda)$ if and only if
\begin{equation}\label{eq:Flambda}
\operatorname{Re}\left(1+\frac{zf''(z)}{f'(z)}\right)
>
\frac{1}{2}-\lambda,
\qquad z\in\mathbb{D}.
\end{equation}
\end{defi}

Clearly, when $-\frac{1}{2}<\lambda\leq\frac{1}{2}$, the class $\mathcal{F}(\lambda)$ forms a subclass of the class $\mathcal{C}$ of convex
functions, with $\mathcal{F}\!\left(\frac{1}{2}\right)=\mathcal{C}.$
Moreover, since $\frac{1}{2}-\lambda\geq-\frac{1}{2}$ for $\frac{1}{2}\leq\lambda\leq 1$, every function in $\mathcal{F}(\lambda)$ is close-to-convex whenever $\frac{1}{2}\leq\lambda\leq 1$.

Although the functions in $\mathcal{F}(\lambda)$ are close-to-convex for $\frac12\leq\lambda\leq1$, Pfaltzgraff, Reade, and Umezawa~\cite{Pfaltzgraff1976} showed that the classes $\mathcal{F}(\lambda)$ contain non-starlike functions for every $\frac{1}{2}<\lambda\leq 1$. Furthermore, Umezawa \cite{Umezawa1952} proved that functions in $\mathcal{F}(1)$ are convex in one direction. Following the terminology in the literature, we call the functions in $\mathcal{F}(\lambda)$, for $\frac{1}{2}\leq\lambda\leq1$, the
\emph{Ozaki close-to-convex functions} and denote this class by
$\mathcal{F}_{O}(\lambda)$. \par

\medskip
The following definition extends the class of Ozaki convex to convex mappings of type B by introducing a spiral parameter.

\begin{defi}
Let \(X\) be a complex Banach space, \(\mathbb{B}\) be the unit ball of \(X\), and let
\(F:\mathbb{B}\to X\) be a normalized locally biholomorphic mapping.
Then \(F\) is called a \emph{Ozaki convex to convex mapping of type B} on
\(\mathbb{B}\) if
\[
\operatorname{Re}\left\{
T_x\!\left(\frac{(DF(x))^{-1}\bigl(D^2F(x)(x,x)+DF(x)x\bigr)}{\parallel x\parallel} 
\right)
\right\}\geq \frac{1}{2}-\lambda,
\]
for $\frac{1}{2}\leq \lambda \leq 1$, every \(x\in\mathbb{B}\setminus\{0\}\) and every \(T_x\in T(x)\). We denote this class by $\widehat{\mathcal{F}}_{\lambda}(\mathbb{B})$.
\end{defi}
\begin{rem}
The class \(\widehat{\mathcal{F}}_{\lambda}(\mathbb{B})\) of Ozaki convex to convex mappings of type~B is a natural extension of the classical Ozaki class to complex Banach spaces. Indeed, when \(X=\mathbb{C}\) and \(\mathbb{B}=\mathbb{U}=\{z\in\mathbb{C}:|z|<1\}\), the supporting functional is given by
\[
T_z(w)=\frac{w}{|z|}, \qquad z\neq 0,
\]
and hence
\[
\frac{(Df(z))^{-1}\bigl(D^{2}f(z)(z,z)+Df(z)z\bigr)}{|z|}= 1+\frac{zf''(z)}{f'(z)}.
\]
Consequently, the defining condition of \(\widehat{\mathcal{F}}_{\lambda}(\mathbb{B})\) reduces to
\[
\operatorname{Re}\left(1+\frac{zf''(z)}{f'(z)}\right)\geq\frac{1}{2}-\lambda, \qquad \frac{1}{2}\le\lambda\le 1,
\]
which is precisely the defining condition of the classical Ozaki close-to-convex class \(\mathcal{F}_0(\lambda)\). Thus, \(\widehat{\mathcal{F}}_{\lambda}(\mathbb{B})\) provides a natural higher-dimensional analogue of the Ozaki class in the setting of complex Banach spaces.
\end{rem}
We first recall the definitions of the Hankel, Toeplitz and Hermitian--Toeplitz determinants for functions in the class $\mathcal{A}$.

\begin{defi}
Let \(f\in\mathcal{A}\) be given by (\ref{eq:A}). Then, for integers \(q\geq 1\) and \(n\geq 0\), the \(q\)th Hankel determinant is defined by
\[
H_q(n)(f):=
\begin{vmatrix}
a_n & a_{n+1} & \cdots & a_{n+q-1}\\
a_{n+1} & a_{n+2} & \cdots & a_{n+q}\\
\vdots & \vdots & \ddots & \vdots\\
a_{n+q-1} & a_{n+q} & \cdots & a_{n+2q-2}
\end{vmatrix}.
\]
In particular,
\[
H_2(2)(f)=a_2a_4-a_3^2.
\]
\end{defi}

\begin{defi}
Let \(f\in\mathcal{A}\) be given by (\ref{eq:A}). Then, for integers $q\ge1$ and $n\ge0$, the \emph{$q$th Toeplitz determinant} is defined by
\[
T_q(n)(f):=
\begin{vmatrix}
a_n & a_{n+1} & \cdots & a_{n+q-1}\\
a_{n+1} & a_n & \cdots & a_{n+q-2}\\
\vdots & \vdots & \ddots & \vdots\\
a_{n+q-1} & a_{n+q-2} & \cdots & a_n
\end{vmatrix}.
\]
In particular,
\[
T_3(1)(f)
=
\begin{vmatrix}
1 & a_2 & a_3\\
a_2 & 1 & a_2\\
a_3 & a_2 & 1
\end{vmatrix}
=
1-2a_2^2+2a_2^2a_3-a_3^2,
\]
and
\[
T_3(2)(f)
=
\begin{vmatrix}
a_2 & a_3 & a_4\\
a_3 & a_2 & a_3\\
a_4 & a_3 & a_2
\end{vmatrix}
=
a_2^3-2a_2a_3^2+2a_3^2a_4-a_2a_4^2.
\]
\end{defi}

\begin{defi}
Let \(f\in\mathcal{A}\) be given by (\ref{eq:A}). Then, for integers $q\ge1$ and $n\ge0$, the \emph{$q$th Hermitian--Toeplitz determinant} is defined by
\[
T_{q,n}(f):=
\begin{vmatrix}
a_n & a_{n+1} & \cdots & a_{n+q-1}\\
\overline{a_{n+1}} & a_n & \cdots & a_{n+q-2}\\
\vdots & \vdots & \ddots & \vdots\\
\overline{a_{n+q-1}} & \overline{a_{n+q-2}} & \cdots & a_n
\end{vmatrix},
\]
where $\overline{a_k}$ denotes the complex conjugate of $a_k$. In the special case when all coefficients $a_n$ are real, $T_{q,n}(f)$ reduces to the ordinary Toeplitz determinant.

In particular,
\[
T_{3,1}(f)
=
\begin{vmatrix}
1 & a_2 & a_3\\
\overline{a_2} & 1 & a_2\\
\overline{a_3} & \overline{a_2} & 1
\end{vmatrix}
=
1-2|a_2|^2+2\operatorname{Re}(a_2^2\overline{a_3})-|a_3|^2.
\]
\end{defi}

Pommerenke \cite{Pommerenke1966} established a fundamental result for the class $\mathcal{S}$, which stimulated extensive research on analogous coefficient problems for various subclasses of univalent functions. More recently, considerable effort has
been devoted to obtaining sharp estimates for the second Hankel determinant $H_{2,2}(f)=a_2a_4-a_3^2,$ and several significant results have been reported in the literature (see, for example, \cite{CKKLS2018,CKKLS2017,XDL2026,XHX2025}). \\
In contrast, the study of Toeplitz determinants has emerged only in recent years, beginning with the work in \cite{ATA2018}. Along similar lines, Hermitian--Toeplitz determinants were first investigated in \cite{CKLS2020}.

We now turn our attention to the Zalcman functional, its connection with the celebrated Zalcman conjecture and the generalized formulation introduced by Ma \cite{M1999}. \\
In the early 1970s, Lawrence Zalcman conjectured that if \(f\in\mathcal{S}\) is of the form (\ref{eq:A}), then
\[
\left|a_n^2-a_{2n-1}\right|\le (n-1)^2,\qquad n\ge 2.
\]
Moreover, equality is attained by the Koebe function $k(z)=\frac{z}{(1-z)^2},\;\; z\in\mathbb{D},$
and its rotations.

\begin{defi}
Let $f(z)=z+\sum_{n=2}^{\infty}a_nz^n\in\mathcal{A}.$
For integers \(m,n\ge2\), the \emph{generalized Zalcman functional} is defined by
$J_{m,n}(f):=a_ma_n-a_{m+n-1}.$
In particular, $J_{2,3}(f)=a_2a_3-a_4.$
\end{defi}

Ma~\cite{M1999} conjectured that if \(f\in\mathcal{S}\), then
\[
|J_{m,n}(f)|\le (m-1)(n-1), \qquad m,n\ge2.
\]
Furthermore, he proved this conjecture for the class \(\mathcal{S}^{*}\), and also for functions in \(\mathcal{S}\) whose Taylor coefficients are all real.\\

More recently, in 2022, Allu, Lecko and Thomas \cite{ALT2022} established sharp bounds for the second Hankel determinant, several Toeplitz determinants and several Hermitian--Toeplitz determinants for functions belonging to the class of Ozaki close-to-convex functions. They also obtained a sharp estimate for the generalized Zalcman functional $J_{2,3}(f)$. Their results are summarized in the following theorems.

\begin{theoA}\cite[Theorem 3.1]{ALT2022}
Let \(f\in \mathcal{F}_{O}(\lambda)\), where $\frac{1}{2}\le \lambda\le 1.$
Then
\[
\left|H_2(2)(f)\right|
\le
\frac{(1+2\lambda)^2(17-10\lambda)}
{192(3-2\lambda)}.
\]
Moreover, the inequality is sharp.
\end{theoA}
\begin{theoB}\cite[Theorem 4.2]{ALT2022}
Let \(f\in\mathcal{F}_{O}(\lambda)\), where $\frac{1}{2}\le \lambda\le 1.$
Then
\[
\left|T_3(1)(f)\right|
\le
\frac{(4+3\lambda+2\lambda^2)(7+6\lambda+8\lambda^2)}{18}.
\]
Moreover, the inequality is sharp.
\end{theoB}
\begin{theoC}\cite[Theorem 4.3]{ALT2022}
Let \(f\in\mathcal{F}_{O}(\lambda)\), where $\frac{1}{2}\le \lambda\le 1.$
Then
\[
\left|T_3(2)(f)\right|
\le
\frac{(1+2\lambda)^3(9+5\lambda+2\lambda^2)(25+17\lambda+10\lambda^2)}{864}.
\]
Moreover, the inequality is sharp.
\end{theoC}
\begin{theoD}\cite[Theorem 5.1]{ALT2022}
Let \(f\in\mathcal{F}_{O}(\lambda)\), where $\frac{1}{2}\le \lambda\le 1.$
Then
\[
-\frac{(2\lambda-1)^2(2\lambda+5)^2}{64\lambda(2\lambda+3)}\leq T_{3,1}(f)\leq  \begin{cases} 1 &\lambda\in [1/2, (\sqrt{153}-5)/8]\\
1+\frac{(1+2\lambda)^2(4\lambda^2+5\lambda-8)}{18}, &\lambda\in ((\sqrt{153}-5)/8,1].
\end{cases}
\]
 The inequalities are sharp.
\end{theoD}
\begin{theoE}\cite[Theorem 6.1]{ALT2022}
Let \(f\in\mathcal{F}_{O}(\lambda)\), where $\frac{1}{2}\le \lambda\le 1.$
Then
\[
|J_{2,3}(F)|\leq \frac{(1+2\lambda)(7+2\lambda)^{3/2}}{36\sqrt{3(9-4\lambda^2)}}.
\]
 The inequalities are sharp.
\end{theoE}

Many authors have investigated various subclasses of biholomorphic mappings in Banach spaces and their associated coefficient problems. Significant contributions in this direction can be found in \cite{Chirila2014, XL2009,GHK2002,Kohr1998}, \cite{GHKK2017}--\cite{HKK2021}.

\medskip

The above results naturally raise the following question: 
\begin{ques}Can the sharp bounds for the second Hankel determinant, Toeplitz determinants, Hermitian--Toeplitz determinants, and the Zalcman functional obtained for functions in $\mathcal{F}_{O}(\lambda)$ be extended to the setting of normalized biholomorphic mappings on complex Banach spaces?
\end{ques}
The primary objective of the present paper is to answer this question by establishing Banach space analogues of these determinant estimates.

The paper is organized into six sections. In Section~2, we establish several
auxiliary lemmas that are required for the proofs of our main results.
In Section~3, we determine the second Hankel determinant
$H_{2}(2)(f)$ for functions belonging to the class
$\widehat{\mathcal{F}}_{\lambda}(\mathbb{B})$. In Section~4, we investigate
the Toeplitz determinants $T_{3}(1)(f)$ and $T_{3}(2)(f)$ for the class
$\widehat{\mathcal{F}}_{\lambda}(\mathbb{B})$. In Section~5, we obtain
estimates for the Hermitian--Toeplitz determinant $T_{3,1}(f)$ for functions
in the class $\widehat{\mathcal{F}}_{\lambda}(\mathbb{B})$. Finally, in
Section~6, we study the Zalcman functional $J_{2,3}(F)$ for functions
belonging to the class $\widehat{\mathcal{F}}_{\lambda}(\mathbb{B})$.
\section{{\bf Lemmas}}
Before proving of the main theorem, we establish the following auxiliary lemmas. 
\begin{lem}\label{L0}\cite[Lemma 2.4]{ALT2022}
Suppose that $f\in\mathcal{F}_{0}(\lambda)$ and is given by \eqref{eq:A}. Then
\begin{equation}\label{eq:coefficients}
\begin{aligned}
a_2 &= \frac{1}{4}p_1(1+2\lambda),\\[2mm]
a_3 &= \frac{1}{12}\left((1+2\lambda)p_2+\frac{1}{2}(1+2\lambda)^2p_1^2\right),\\[2mm]
a_4 &= \frac{1}{24}\left((1+2\lambda)p_3
+\frac{3}{4}(1+2\lambda)^2p_1p_2
+\frac{1}{8}(1+2\lambda)^3p_1^3\right),
\end{aligned}
\end{equation}
where $p_1$, $p_2$, and $p_3$ are given by \eqref{eq:P}.
\end{lem}

\begin{lem}\label{L1}
Let \(u\in X\) with \(\|u\|=1\) and let \(T_u\in T(u)\). Define $F(x)=\frac{f(T_u(x))}{T_u(x)}\,x,$ $x\in\mathbb{B},$
where \(f\in\mathcal{S}\). If \(f\) is a normalized Ozaki close to convex function on \(\mathbb{U}\), then \(F\) is a Ozaki close to convex  mapping of type \(B\) on \(\mathbb{B}\).
\end{lem}

\begin{proof}
Since $f$ is a normalized Ozaki close to convex  function on $\mathbb{U}$, we have
\begin{equation}
\operatorname{Re}\left(1+\frac{\xi f''(\xi)}{f'(\xi)}\right)>\frac{1}{2}-\lambda, \frac{1}{2}\leq \lambda \leq 1
\qquad \xi\in\mathbb{U}.
\label{eq:spiral}
\end{equation}

For convenience, let $h(x)=\frac{f(T_u(x))}{T_u(x)},\; x\in\mathbb{B}.$ It follows immediately from the definition of $F$ that $F(x)=h(x)x.$

The first and second order Fr\'echet derivatives of $F$ are given by
\[
DF(x)(x)=h(x)x+Dh(x)(x)\,x\;\text{and}\; D^2F(x)(x,x)=D^2h(x)(x,x)x+2Dh(x)(x)\qquad x\in X.
\]
From above two equations we have 
\bea\label{le.1}
D^2F(x)(x,x)+DF(x)x
=
\left(h(x)+3Dh(x)x+D^2h(x)(x,x)\right)x,
\eea
Moreover, by the chain rule for Fr\'echet derivatives, we obtain
\bea\label{mmm1}
&& h(x)+Dh(x)x=f'(T_u(x))\\
\label{mmm2}\text{and}\;&& h(x)+3Dh(x)x+D^2h(x)(x,x)=T_u(x)f''(T_u(x))+f'(T_u(x)).
\eea

A straightforward calculation yields
\bea\label{le.2}
(DF(x))^{-1}\eta
=
\frac1{h(x)}
\left(
\eta-\frac{(Dh(x)\eta) x}
{h(x)+Dh(x)x}
\right).
\eea
Substituting $\eta=D^2F(x)(x,x)+DF(x)x$ into \eqref{le.2} and combining the resulting expression with \eqref{le.1}, we obtain

\bea\label{le.3}
&&(DF(x))^{-1}(D^2F(x)(x,x)+DF(x)x)\nonumber\\
&=&
\frac{1}{h(x)}\left((D^2F(x)(x,x)+DF(x)x)-\frac{Dh(x)(D^2F(x)(x,x)+DF(x)x)x}
{h(x)+Dh(x)x}\right)\nonumber\\
&=&\frac{1}{h(x)}\bigg(D^2h(x)(x,x)x+3Dh(x)(x)+h(x)x
\nonumber\\
&&\quad\quad-\frac{Dh(x)(D^2h(x)(x,x)x+3Dh(x)(x)+h(x)x)}
{h(x)+Dh(x)x}x\bigg).\eea

Since \(h(x)+3Dh(x)x+D^2h(x)(x,x)\) is a scalar say $H(x)$, it follows that
\[
Dh(x)H(x)x
=H(x)Dh(x)x.
\]
Hence from (\ref{le.3}) we get
\bea\label{kkk1}
&&(DF(x))^{-1}\left(D^2F(x)(x,x)+DF(x)x\right)\nonumber\\
&=&
\frac{h(x)+3Dh(x)x+D^2h(x)(x,x)}{h(x)}
\left(
1-\frac{Dh(x)x}{h(x)+Dh(x)x}
\right)x\nonumber\\
&=&
\frac{h(x)+3Dh(x)x+D^2h(x)(x,x)}
{h(x)+Dh(x)x}\,x.
\eea
Applying the identities \eqref{mmm1} and \eqref{mmm2} to \eqref{kkk1}, we obtain
\[
(DF(x))^{-1}\left(D^2F(x)(x,x)+DF(x)x\right)
=\left(1+\frac{T_u(x)f''(T_u(x))}{f'(T_u(x))}\right)x.
\]
Therefore,
\bea\label{f1}
T_x\!\left((DF(x))^{-1}\left(D^2F(x)(x,x)+DF(x)x\right)\right)=\left(1+\frac{T_u(x)f''(T_u(x))}{f'(T_u(x))}\right)\|x\|.
\eea

Combining \eqref{f1} with \eqref{eq:spiral}, we deduce that
\beas
&&\operatorname{Re}\left(T_x\!\left(\frac{(DF(x))^{-1}\left(D^2F(x)(x,x)+DF(x)x\right)}{\|x\|}\right)\right)\\
&=&\operatorname{Re}\;\left(1+\frac{T_u(x)f''(T_u(x))}{f'(T_u(x))}\right)>\frac{1}{2}-\lambda,\;\;\frac{1}{2}\leq \lambda\leq 1,
\eeas
for $x\in \mathbb{B}\setminus\{0\}$ Hence, by the definition of Ozaki close to convex mapping of type \(B\) on \(\mathbb{B}\), $F$ is a Ozaki close to convex mapping of type \(B\) on \(\mathbb{B}\).
\end{proof}

\begin{lem}\label{L2}
Suppose that $g\in H(\mathbb{B},\mathbb{C})$ with $g(0)=1$, and define $F(x)=g(x)x, x\in\mathbb{B}.$ Fix $x_{0}\in\partial\mathbb{B}$ and let
$f(\xi)=g(\xi x_{0})\,\xi,\xi\in\mathbb{U}.$
Then,
\[
f\in\widehat{\mathcal{F}}_{O}(\lambda)
\quad\Longleftrightarrow\quad
F\in\widehat{\mathcal{F}}_{\lambda}(\mathbb{B}).
\]
\end{lem}

\begin{proof}
Assume first that $F\in\widehat{\mathcal{F}}_{\lambda}(\mathbb{B}).$ Then $F$ is locally biholomorphic on $\mathbb{B}$. Consequently,
\[
g(x)\neq0,
\qquad
g(x)+Dg(x)x\neq0,
\qquad x\in\mathbb{B}.
\]

Since
\[
DF(x)(x)=g(x)x+Dg(x)(x)\,x,
\]
a direct computation gives
\begin{equation}
[DF(x)]^{-1}=\frac1{g(x)}\left(I-\frac{xDg(x)}{g(x)+Dg(x)x}\right).
\label{eq:inverse}
\end{equation}
Since $F(x)=g(x)x,$ we have
\[
DF(x)h=g(x)h+Dg(x)(h)x,\;\;\text{and}\;\;
D^2F(x)(h,k)=D^2g(x)(h,k)x+Dg(x)(h)k+Dg(x)(k)h.
\]

Therefore,
\[
D^2F(x)(x,x)=D^2g(x)(x,x)x+2Dg(x)(x)x,
\]
and hence
\[D^2F(x)(x,x)+DF(x)x=\bigl(g(x)+3Dg(x)x+D^2g(x)(x,x)\bigr)x.
\]

Substituting this identity into \eqref{eq:inverse}, we obtain
\[
\begin{aligned}
&(DF(x))^{-1}\bigl(D^2F(x)(x,x)+DF(x)x\bigr) \\
&=\frac{g(x)+3Dg(x)x+D^2g(x)(x,x)}{g(x)}\left(I-\frac{xDg(x)}{g(x)+Dg(x)x}
\right)x.
\end{aligned}
\]

Since
\[\left(I-\frac{xDg(x)}{g(x)+Dg(x)x}\right)x=x-\frac{Dg(x)x}{g(x)+Dg(x)x}x
=\frac{g(x)}{g(x)+Dg(x)x}x,\]
it follows that
\[
\begin{aligned}
(DF(x))^{-1}\bigl(D^2F(x)(x,x)+DF(x)x\bigr)
&=\frac{g(x)+3Dg(x)x+D^2g(x)(x,x)}{g(x)+Dg(x)x}\,x \\
&=\left(1+\frac{2Dg(x)x+D^2g(x)(x,x)}{g(x)+Dg(x)x}\right)x.
\end{aligned}
\]
Substituting $x=\xi x_0$ into above equation, we obtain
\[
\begin{aligned}
&(DF(\xi x_0))^{-1}\bigl(D^2F(\xi x_0)(\xi x_0,\xi x_0)+DF(\xi x_0)\xi x_0\bigr)\\
&=\left(1+\frac{2Dg(\xi x_0)\;\xi x_0+D^2g(\xi x_0)(\xi x_0,\xi x_0)}
{g(\xi x_0)+Dg(\xi x_0)\;\xi x_0}
\right)\xi x_0.
\end{aligned}
\]

Applying $T_{\xi x_0}$, we get
\[
\begin{aligned}&e^{i\beta}T_{\xi x_0}\left((DF(\xi x_0))^{-1}\bigl(D^2F(\xi x_0)(\xi x_0,\xi x_0)+DF(\xi x_0)\xi x_0\bigr)\right)\\
&=e^{i\beta}\left(1+\frac{2Dg(\xi x_0)\;\xi x_0+D^2g(\xi x_0)(\xi x_0,\xi x_0)}{g(\xi x_0)+Dg(\xi x_0)\;\xi x_0}\right)T_{\xi x_0}(\xi x_0).
\end{aligned}
\]

Since $T_{\xi x_0}(\xi x_0)=|\xi|,$ and from $f(\xi)=g(\xi x_{0})\xi$ we have 
\[
f'(\xi)=g(\xi x_0)+Dg(\xi x_0),\xi x_0\;\text{and}\;f''(\xi)=2Dg(\xi x_0)(x_0)+\xi D^2g(\xi x_0)(x_0,x_0),
\]
we have
\[
2Dg(\xi x_0)\;\xi x_0+D^2g(\xi x_0)(\xi x_0,\xi x_0)=\xi f''(\xi).
\]
Hence,
\bea\label{kkkk1}
T_{\xi x_0}\left(\frac{(DF(\xi x_0))^{-1}\bigl(D^2F(\xi x_0)(\xi x_0,\xi x_0)+DF(\xi x_0)\xi x_0\bigr)}{\|\xi x_0\|}
\right)
=1+\frac{\xi f''(\xi)}{f'(\xi)}.
\eea

Since $F\in\widehat{\mathcal{F}}_{\lambda}(\mathbb{B}),$
it follows that
\beas\label{kkkk1}
\operatorname{Re}\;\left\{1+\frac{\xi f''(\xi)}{f'(\xi)}\right\}>\frac{1}{2}-\lambda
\eeas
and hence $f\in\mathcal{F}_{O}(\lambda).$\\

\medskip

Conversely, suppose that $f\in\mathcal{F}_{O}(\lambda).$ The definition of the class \(\mathcal F_0(\lambda)\) includes that \(f\) is locally univalent in \(\mathbb U\). Hence $f'(\xi)\neq0,\qquad |\xi|<1.$

Therefore,
\[g(x)+Dg(x)x\neq0,
\qquad x\in\mathbb{B}.
\]
By \eqref{eq:inverse}, $DF(x)$ is invertible for every $x\in\mathbb{B}$, and therefore $F$ is locally biholomorphic.

Finally, using \eqref{kkkk1} together with
\[
\operatorname{Re} \left(1+\frac{\xi f''(\xi)}{f'(\xi)}\right)>\frac{1}{2}-\lambda,
\]
we obtain
\[
\operatorname{Re}\left(T_x\left(\frac{(DF(x))^{-1}\bigl(D^2F(x)(x,x)+DF(x)x\bigr)}{\|x\|}\right)\right)>\frac{1}{2}-\lambda, \qquad x\in\mathbb{B}\setminus\{0\}.
\]
Hence, $F\in\widehat{\mathcal{F}}_{\lambda}(\mathbb{B}).$

This completes the proof.
\end{proof}
\begin{lem}\label{L3}  Let $g\in H(\mathbb{B},\mathbb{C})$ satisfy $g(0)=1$, and define $F(x)=g(x)x, x\in\mathbb{B}.$ Suppose that $F\in\widehat{\mathcal{F}}_{\lambda}(\mathbb{B})$. Then, for every
$x_{0}\in X$ with $\|x_{0}\|=1$,
\bea\label{A234}
\begin{cases} A_{2}
&=
\frac{1}{2!}
T_{x_{0}}
\left(
D^{2}F(0)(x_{0}^{2})
\right)
=\frac{1}{4}p_1(1+2\lambda),\\
A_{3}
&=
\frac{1}{3!}
T_{x_{0}}
\left(
D^{3}F(0)(x_{0}^{3})
\right)
=\frac{1}{12}(1+2\lambda)\left(p_2+\frac{1}{2}p_1^2(1+2\lambda)\right),
\\
\text{and}\;\; 
A_{4}
&=
\frac{1}{4!}T_{x_{0}}\left(D^{4}F(0)(x_{0}^{4})\right)=\frac{1}{24}(1+2\lambda)\left(p_3+\frac{3}{4}p_1p_2(1+2\lambda)+\frac{1}{8}p_1^3(1+2\lambda)^2\right),
\end{cases}
\eea
where $p_1$, $p_2$ and $p_3$ are given by (\ref{eq:P}).

\end{lem}

 \begin{proof}
Fix $x_{0}\in\partial\mathbb{B}$ and define $f(\xi)=g(\xi x_{0})\,\xi, \xi\in\mathbb{U}.$
Since $F\in\widehat{\mathcal{F}}_{\lambda}(\mathbb{B})$, Lemma~\ref{L2}
implies that $f\in \mathcal{F}_{0}(\lambda).$

On the other hand,
\[
f(\xi)
=
T_{x_{0}}\bigl(F(\xi x_{0})\bigr),
\]
and hence
\[
a_{2}
=
\frac{f''(0)}{2!}
=
\frac{1}{2!}
T_{x_{0}}
\left(
D^{2}F(0)(x_{0}^{2})
\right)
=
A_{2},
\]
\[
a_{3}
=
\frac{f^{(3)}(0)}{3!}
=
\frac{1}{3!}
T_{x_{0}}
\left(
D^{3}F(0)(x_{0}^{3})
\right)
=
A_{3},
\]
and \[
a_{4}
=
\frac{f^{(4)}(0)}{4!}
=
\frac{1}{4!}
T_{x_{0}}
\left(
D^{4}F(0)(x_{0}^{4})
\right)
=
A_{4}.
\]
Applying Lemma~\ref{L0}, we obtain the desired result.
\end{proof}
\begin{lem}\label{L4} Let $F$ be a locally biholomorphic mapping on $\mathbb{B}$ and $F$ satisfy the following assumptions:
\begin{equation}\label{cc1}
\frac{D^{k+1}F(0)\left(x^{k+1}\right)}{(k+1)!}
= H_{F,k}(x)x,\qquad x\in X,\quad k=1,2,3.
\end{equation}
where $H_{F,k}(x)$ is a homogeneous polynomial of degree $k$ with values in $\mathbb{C}$. If $F\in \widehat{\mathcal{F}}_{\lambda}(\mathbb{B})$, then for each $x_0\in X$ with $\|x_0\|=1$,  the quantities $A_2$, $A_3$ and $A_4$ are given by  (\ref{A234}).
\end{lem}
\begin{proof} Fix $x_0\in \partial B$ and let $T_{x_0}\in T(x_0)$. Define
\bea\label{aaa1}
\phi(\xi)=
\begin{cases}
\dfrac{T_{x_0}\!\left(\varphi(\xi x_0)\right)}{\xi}, & \xi\neq 0,\\[2mm]
1, & \xi=0,
\end{cases}
\eea
where
\[
\varphi(x)=(DF(x))^{-1}\bigl(D^2F(x)(x,x)+DF(x)x\bigl).
\]
Then $\phi\in H(\mathbb U)$ and $\phi(0)=1$. Moreover, since
$F\in \widehat{\mathcal{F}}_{\lambda}(\mathbb{B})$ , it follows from the definition of Ozaki close to convex function that
\[
\Re\!\left(\frac{T_{x_0}\!\left(\varphi(\xi x_0)\right)}{\xi}\right)
=
\Re\!\left(\frac{T_{\xi x_0}\!\left(\varphi(\xi x_0)\right)}{\|\xi x_0\|}\right)>\frac{1}{2}-\lambda,\;\;\text{for}\;\;\frac{1}{2}\leq \lambda \leq 1
\qquad \xi\in\mathbb U\setminus\{0\}.
\]
Hence,
\[\Re\!\left(\phi(\xi)\right)>\frac{1}{2}-\lambda,\;\;\text{for}\;\;\frac{1}{2}\leq \lambda \leq 1\qquad \xi\in\mathbb U.
\]
 Then there exists $p\in \mathcal{P}$ such that 
\bea\label{kp1}\phi(z)=\left(\frac{1}{2}+\lambda\right) p(z)+\frac{1}{2}-\lambda.\eea
Since $\phi$ is holomorphic function on $\mathbb{U}$, then it has Taylor series expansion:
\bea\label{phi1}\phi(\xi)=1+\frac{T_{x_0}(D^2\varphi(0)(x_0^2))}{2!}\xi+\frac{T_{x_0}(D^3\varphi(0)(x_0^3))}{3!}\xi^2+\frac{T_{x_0}(D^4\varphi(0)(x_0^4))}{4!}\xi^3+\cdots,\xi\in \mathbb{U}.\eea
Using (\ref{aaa1}), \eqref{eq:P}, \eqref{kp1} and \eqref{phi1}, a direct computation yields
\bea\label{qq1}
\begin{cases} \frac{T_{x_0}(D^2\varphi(0)(x_0^2))}{2!}&=\left(\frac{1}{2}+\lambda\right)p_1,\\
 \frac{T_{x_0}(D^3\varphi(0)(x_0^3))}{3!}&=\left(\frac{1}{2}+\lambda\right)p_2\\
\frac{T_{x_0}(D^4\varphi(0)(x_0^4))}{4!}&=\left(\frac{1}{2}+\lambda\right)p_3.
\end{cases}
\eea
On the other hand, since $\varphi(x)=(DF(x))^{-1}\bigl(D^2F(x)(x,x)+DF(x)x\bigl)$, it follows that
\bea\label{qqq2}
\begin{cases}
\frac{D^{2}F(0)(x^{2})}{2!}&=\frac{1}{2}\frac{D^{2}\varphi(0)(x^{2})}{2!}\\
 D^{3}F(0)(x^{3})&=\frac{D^{3}\varphi(0)(x^{3})}{3!}+D^{2}F(0)\!\left(x,\,\frac{D^{2}\varphi(0)(x^{2})}{2!}\right)\\
\frac{D^{4}F(0)(x^{4})}{2}&=\frac{D^{4}\varphi(0)(x^{4})}{4!}+D^{2}F(0)\!\left(x,\,\frac{D^{3}\varphi(0)(x^{3})}{3!}\right)\\
&\;\;\;+\frac{1}{2}D^{3}F(0)\!\left(x^{2},\,\frac{D^{2}\varphi(0)(x^{2})}{2!}\right).
\end{cases}\eea

Combining(\ref{qq1}) and (\ref{qqq2}) together with assumption \eqref{cc1}, we derive explicit formulas for the coefficients \(A_2\), \(A_3\)  and \(A_4\) in terms of the parameters \(p_1\), \(p_2\) and  \(p_3\), respectively.

\bea\label{A2}A_2&=&T_{x_0}\!\left(\frac{D^2F(0)(x_0^2)}{2!}\right)=H_{F,1}(x_0)
=\frac{1}{2}\,T_{x_0}\!\left(\frac{D^2\varphi(0)(x_0^2)}{2!}\right)
=\frac{1}{2}\left(\frac{1}{2}+\lambda\right)p_1\nonumber\\
&=&\frac{1}{4}p_1(1+2\lambda)
\eea
and
\[
\frac{D^2\varphi(0)(x_0^2)}{2!}=2\,\frac{D^2F(0)(x_0^2)}{2!}=2H_{F,1}(x_0)x_0
=2A_2x_0
=\frac{1}{2}p_1(1+2\lambda)x_0.
\]

\bea\label{A3} A_3&=&T_{x_0}\!\left(\frac{D^3F(0)(x_0^3)}{3!}\right)
=H_{F,2}(x_0) \nonumber\\
&=&\frac{1}{6}\left(T_{x_0}\!\left(\frac{D^3\varphi(0)(x_0^3)}{3!}\right)
+T_{x_0}\!\left(D^2F(0)\!\left(x_0,\frac{D^2\varphi(0)(x_0^2)}{2!}\right)\right)\right) \nonumber\\
&=&\frac{1}{6}\left(\cos\beta\,e^{-i\beta}p_2+T_{x_0}\!\left(D^2F(0)(x_0,\frac{1}{2}p_1(1+2\lambda)x_0)\right)\right) \nonumber\\
&=&\frac{1}{6}\left(\left(\frac{1}{2}+\lambda\right)p_2+p_1(1+2\lambda)A_2\right)\nonumber\\
&=&\frac{1}{12}(1+2\lambda)\left(p_2+\frac{1}{2}p_1^2(1+2\lambda)\right).
\eea

and

\beas 
\frac{D^3\varphi(0)(x_0^3)}{3!}&=&D^3F(0)(x_0^3)-D^2F(0)\!\left(x_0,\frac{D^2\varphi(0)(x_0^2)}{2!}\right)\\
&=&D^3F(0)(x_0^3)-D^2F(0)(x_0,\frac{1}{2}p_1(1+2\lambda)x_0)\\
&=&\left(6H_{F,2}(x_0)-\frac{1}{4}p_1^2(1+2\lambda)^2\right)x_0\\
&=&\left(\frac{1}{2}+\lambda\right)p_2x_0\eeas
\bea\label{A4}
A_4&=& T_{x_0}\!\left(\frac{D^4F(0)(x_0^4)}{4!}\right)=H_{F,3}(x_0) \nonumber\\
&=&\frac{1}{12}\Bigg(T_{x_0}\!\left(\frac{D^4\varphi(0)(x_0^4)}{4!}\right)+T_{x_0}\!\left(D^2F(0)\!\left(x_0,\frac{D^3\varphi(0)(x_0^3)}{3!}\right)\right) \nonumber\\
&& +\frac{1}{2}T_{x_0}\!\left(D^3F(0)\!\left(x_0^2,\frac{D^2\varphi(0)(x_0^2)}{2!}\right)\right)\Bigg) \nonumber\\
&=&\frac{1}{12}\Bigg(\left(\frac{1}{2}+\lambda\right)p_3+
T_{x_0}\!\left(D^2F(0)(x_0,\left(\frac{1}{2}+\lambda\right)p_2x_0)\right) \nonumber\\
&&\frac{1}{2}T_{x_0}\!\left(D^3F(0)(x_0^2,\left(\frac{1}{2}+\lambda\right)p_1x_0)\right)\Bigg) \nonumber\\
&=&\frac{1}{12}\Bigg(\left(\frac{1}{2}+\lambda\right)p_2+2\left(\frac{1}{2}+\lambda\right)p_2A_2+3\left(\frac{1}{2}+\lambda\right)p_1A_3\Big) \nonumber\\
&=&\frac{1}{24}(1+2\lambda)\left(p_3+\frac{3}{4}p_1p_2(1+2\lambda)+\frac{1}{8}p_1^3(1+2\lambda)^2\right).\eea
\end{proof}
\section{{\bf The second Hankel determinant $H_2(2)(f)$ for class $\widehat{\mathcal{F}}_{\lambda}(\mathbb{B})$.}}
In this section, we compute sharp upper bounds for 
$H_{2}(2)(F)=
\begin{vmatrix}
A_{2} & A_{3} \\
A_{3} & A_{4} \\
\end{vmatrix}$
or equivalently, $H_{2}(2)(F)
=A_2A_4-A_3^2,$
with  $A_{2},A_{3},A_{4}$ defined by
\eqref{eq:An},
over the class $\widehat{\mathcal{F}}_{\lambda}(\mathbb{B})$.

\begin{theo}\label{H1}
Let $g\in H(\mathbb{B},\mathbb{C})$ satisfy $g(0)=1$, and define $F(x)=g(x)x, x\in\mathbb{B}.$
Suppose that $F\in\widehat{\mathcal{F}}_{\lambda}(\mathbb{B})$. Then, for every
$x_{0}\in X$ with $\|x_{0}\|=1$,
\[
|H_{2}(2)(F)|\leq  \frac{(1+2\lambda)^2(17-10\lambda)}{192(3-2\lambda)}.
\]
 The inequality is sharp.
\end{theo}
\begin{proof} Fix \(x_{0}\in\partial\mathbb{B}\) and define
\[
f(\xi)=g(\xi x_{0})\,\xi,\qquad \xi\in\mathbb{U}.
\]
Since \(F\in\widehat{\mathcal{F}}_{\lambda}(\mathbb{B})\), it follows from Lemmas~\ref{L2} and \ref{L3} that
\begin{align*}
|H_{2}(2)(F)|
&=|A_2A_4-A_3^2|\\
&=
\left|
-\frac{(1+2\lambda)^4p_1^4}{2304}
+\frac{(1+2\lambda)^3p_1^2p_2}{1152}
-\frac{(1+2\lambda)^2p_2^2}{144}
+\frac{(1+2\lambda)^2p_1p_3}{96}
\right|.
\end{align*}
Proceeding as in the proof of Theorem~A (see \cite[Theorem~3.1]{ALT2022}), we obtain the desired result.
\end{proof}

The following example demonstrates that the estimate obtained in Theorem \ref{H1} is sharp.

\begin{exm}
Consider the mapping
\[
F(x)=
\frac{f\!\left(T_{x_0}(x)\right)}{T_{x_0}(x)}\,x,
\qquad
x_0\in\partial\mathbb{B},\;
T_{x_0}\in T(x_0),
\]
where
\[
f(z)=\int_0^{z}(1-t_0t+t^2)^{-(\frac{1}{2}+\lambda)}dt,
\qquad z\in\mathbb{U},
\]
where $t_0=\sqrt{\frac{2}{3-2\lambda}}$.
Lemma~\ref{L1} yields 
\[\frac{D^nF(0)(x^n)}{n!}=a_nT_{x_0}^{n-1}(x)x,\]
where $a_n$ is the coefficient of $z^n$ in the Taylor expansion of $f$.
Furthermore, a straightforward computation gives
\beas
\frac{D^2F(0)(x^2)}{2!}&=&\frac{(1+2\lambda)}{4}\sqrt{\frac{2}{3-2\lambda}}T_{x_0}(x)x,\;\;\;
\frac{D^3F(0)(x^3)}{3!}
=\frac{(1+2\lambda)(2\lambda-1)}{12(3-2\lambda)}
\bigl(T_{x_0}(x)\bigr)^2x,\\
\;\;\;\text{and}\;\;\;
\frac{D^4F(0)(x^4)}{4!}&=&\frac{\sqrt2\left(56\lambda^3+108\lambda^2-62\lambda-39\right)}{96(3-2\lambda)^{3/2}}\bigl(T_{x_0}(x)\bigr)^3x.
\eeas

Hence,
\[
A_2=\frac{(1+2\lambda)}{4}\sqrt{\frac{2}{3-2\lambda}},\;\;
A_3=\frac{(1+2\lambda)(2\lambda-1)}{12(3-2\lambda)},\;\;\text{and}\;\;
A_4=\frac{\sqrt2\left(56\lambda^3+108\lambda^2-62\lambda-39\right)}{96(3-2\lambda)^{3/2}}.
\]
Therefore,
\[
|H_{2}(2)(F)|
=
\left|A_2A_4-A_3^2\right|
=
 \frac{(1+2\lambda)^2(17-10\lambda)}{192(3-2\lambda)}.\]

Thus, the upper bound obtained in Theorem \ref{H1} is attained, showing that the estimate is sharp.
\end{exm}

\medskip

Next, removing the restrictive assumption $F(x)=g(x)x$, we generalize Theorem~A to higher dimensions under weaker assumptions than those of Theorem \ref{H1}. Assume that
\bea\label{cc1}
\frac{D^{k+1}F(0)\left(x^{k+1}\right)}{(k+1)!}
= H_{F,k}(x)x,\qquad x\in X,\quad k=1,2,3.
\eea
where $H_{F,k}(x)$ is a homogeneous polynomial of degree $k$ with values in $\mathbb{C}$. The fact that the assumption \eqref{cc1} is weaker than that of Theorem \ref{H1} is justified in \cite{EJ2024,H2023}.

\begin{theo}\label{H2}
Let $F$ be a locally biholomorphic mapping on $\mathbb{B}$, and suppose that $F$ satisfies the assumption \eqref{cc1}. If $F\in \widehat{\mathcal{F}}_{\lambda}(\mathbb{B})$, then for every $x_0\in X$ with $\|x_0\|=1$, we have
\[
|H_{2}(2)(F)|\leq  \frac{(1+2\lambda)^2(17-10\lambda)}{192(3-2\lambda)}.
\] The inequalities are sharp.
\end{theo}
\begin{proof} Fix $x_0\in \partial{\mathbb{B}}$ and let $T_{x_0}\in T(x_0)$. Then applying Lemma \ref{L4} we deduce that 

\begin{align*}
|H_{2}(2)(F)|
&=|A_2A_4-A_3^2|\\
&=
\left|
-\frac{(1+2\lambda)^4p_1^4}{2304}
+\frac{(1+2\lambda)^3p_1^2p_2}{1152}
-\frac{(1+2\lambda)^2p_2^2}{144}
+\frac{(1+2\lambda)^2p_1p_3}{96}
\right|.
\end{align*}
The remaining part of the proof proceeds exactly as in the proof of Theorem A (see \cite[Theorem 3.1]{ALT2022}); hence, the details are omitted. The sharpness of the result is an immediate consequence of Theorem \ref{H1}. This completes the proof.\end{proof}
\section{{\bf The Toeplitz determinants $T_3(1)(f)$ and $T_3(2)(f)$ for class $\widehat{\mathcal{F}}_{\lambda}(\mathbb{B})$.}}
In this section, we compute sharp upper bounds for 
\[
T_3(1)(F)
=
\begin{vmatrix}
1 & A_2 & A_3\\
A_2 & 1 & A_2\\
A_3 & A_2 & 1
\end{vmatrix}\;\;
\text{and}\;\;
T_3(2)(F)
=
\begin{vmatrix}
A_2 & A_3 & A_4\\
A_3 & A_2 & A_3\\
A_4 & A_3 & A_2
\end{vmatrix}
\]
with  $A_{2},A_{3},A_{4}$ defined by
\eqref{eq:An},
over the class $\widehat{\mathcal{F}}_{\lambda}(\mathbb{B})$.
\begin{theo}\label{T1}
Let $g\in H(\mathbb{B},\mathbb{C})$ satisfy $g(0)=1$, and define $F(x)=g(x)x, x\in\mathbb{B}.$
Suppose that $F\in\widehat{\mathcal{F}}_{\lambda}(\mathbb{B})$. Then, for every
$x_{0}\in X$ with $\|x_{0}\|=1$,
\[
|T_3(1)(f)|\leq  \frac{(4+3\lambda+2\lambda^2)(7+6\lambda+8\lambda^2)}{18}
\]
 and
 \[ T_3(2)(f)|\leq \frac{1}{864}(1+2\lambda)^3(9+5\lambda+2\lambda^2)(25+17\lambda+10\lambda^2).
\]
 The inequalities are sharp.
\end{theo}
\begin{proof} Fix $x_0\in \partial{\mathbb{B}}$ and let $T_{x_0}\in T(x_0)$. Then applying Lemma \ref{L3} we deduce that 

\begin{align*}
|T_3(1)(F)|=
&=|1-2A_2^2+2A_2^2A_3-A_3^2|\\
&=|1-2a_2^2+2a_2^2a_3-a_3^2|\\
&=|T_3(1)(f)|
\end{align*}
and 
\begin{align*}
|T_3(2)(F)|
&=\left|(A_2-A_4)(A_2^2-2A_3^2+A_2A_4)\right|\\
&=\left|(a_2-a_4)(a_2^2-2a_3^2+a_2a_4)\right|\\
&=|T_3(2)(f)|.
\end{align*}
Hence, from Theorems B and C (see \cite[Theorem 4.2, Theorem 4.3]{ALT2022}), we obtain that
\[ |T_3(1)(F)|=|T_3(1)(f)|\leq \frac{(4+3\lambda+2\lambda^2)(7+6\lambda+8\lambda^2)}{18}\]
and 
\[ |T_3(2)(F)|=|T_3(2)(f)| \leq \frac{1}{864}(1+2\lambda)^3(9+5\lambda+2\lambda^2)(25+17\lambda+10\lambda^2).\]
\end{proof}
The following example demonstrates that the estimates obtained in Theorem \ref{T1} are sharp.

\begin{exm}
Consider the mapping
\[
F(x)=
\frac{f\!\left(T_{x_0}(x)\right)}{T_{x_0}(x)}\,x,
\qquad
x_0\in\partial\mathbb{B},\;
T_{x_0}\in T(x_0),
\]
where
\[
f(z)=\frac{1}{2i\lambda}\left(\frac{1}{(1-iz)^{2\lambda}}-1\right),
\qquad z\in\mathbb{U},
\]
By Lemma~\ref{L2} we have $F\in \widehat{\mathcal{F}}_{\lambda}(\mathbb{B})$.

Furthermore, a straightforward computation gives
\beas
\frac{D^2F(0)(x^2)}{2!}&=&\frac{i}{2}(1+2\lambda)T_{x_0}(x)x,\;\;\;
\frac{D^3F(0)(x^3)}{3!}
=-\frac{1}{3}(1+\lambda)(1+2\lambda)\bigl(T_{x_0}(x)\bigr)^2x,\\
\;\;\;\text{and}\;\;\;
\frac{D^4F(0)(x^4)}{4!}&=&-\frac{1}{12}(1+\lambda)(1+2\lambda)(3+2\lambda)\bigl(T_{x_0}(x)\bigr)^3x.
\eeas

Hence,
\[
A_2=\frac{i}{2}(1+2\lambda),\;\;
A_3=-\frac{1}{3}(1+\lambda)(1+2\lambda),\;\;\text{and}\;\;
A_4=-\frac{1}{12}(1+\lambda)(1+2\lambda)(3+2\lambda).
\]
Therefore,
\[
|T_{3}(1)(F)|
=
\left|1-2A_2^2+2A_2^2A_3-A_3^2\right|
=\frac{(4+3\lambda+2\lambda^2)(7+6\lambda+8\lambda^2)}{18}.\]
and 
\[
|T_{3}(2)(F)|
=
\left|(A_2-A_4)(A_2^2-2A_3^2+A_2A_4)\right|
=\frac{1}{864}(1+2\lambda)^3(9+5\lambda+2\lambda^2)(25+17\lambda+10\lambda^2).\]

Thus, the upper bound obtained in Theorem \ref{T1} is attained, showing that the estimates are sharp.
\end{exm}
Next, removing the restrictive assumption $F(x)=g(x)x$, we generalize Theorems B and C to higher dimensions under weaker assumptions than those of Theorem \ref{T1}. 
\begin{theo}\label{T2}
Let $F$ be a locally biholomorphic mapping on $\mathbb{B}$, and suppose that $F$ satisfies the assumption \eqref{cc1}. If $F\in \widehat{\mathcal{F}}_{\lambda}(\mathbb{B})$, then for every $x_0\in X$ with $\|x_0\|=1$, we have
\[
|T_{3}(1)(F)|\leq  \frac{(4+3\lambda+2\lambda^2)(7+6\lambda+8\lambda^2)}{18}.
\]
and 
 \[ T_3(2)(f)|\leq \frac{1}{864}(1+2\lambda)^3(9+5\lambda+2\lambda^2)(25+17\lambda+10\lambda^2),
\]
 The inequalities are sharp.
\end{theo}
\begin{proof} Fix $x_0\in \partial{\mathbb{B}}$ and let $T_{x_0}\in T(x_0)$. Next, applying Lemma \ref{L4} and proceeding analogously to the proofs of Theorems B and C (see \cite[Theorems 4.1 and 4.3]{ALT2022}), we deduce that

\begin{align*}
|T_{3}(1)(F)|
&=|1-2A_2^2+2A_2^2A_3-A_3^2|\\
&\leq \frac{(4+3\lambda+2\lambda^2)(7+6\lambda+8\lambda^2)}{18}
\end{align*}
and 
\beas |T_3(2)(f)| &=& \left|(A_2-A_4)(A_2^2-2A_3^2+A_2A_4)\right|\\
&\leq & \frac{1}{864}(1+2\lambda)^3(9+5\lambda+2\lambda^2)(25+17\lambda+10\lambda^2).
\eeas
 The sharpness of the result is an immediate consequence of Theorem \ref{T1}. This completes the proof.\end{proof}

\section{{\bf The Hermitian-Toeplitz determinant $T_{3,1}(f)$ for class $\widehat{\mathcal{F}}_{\lambda}(\mathbb{B})$.}}
In this section, we compute sharp lower and upper bounds for 
\[T_{3,1}(F)=
\begin{vmatrix}
1 & A_{2} &A_3\\
\overline{A_{2}} & 1 & A_2 \\
\overline{A_3}& \overline{A_2} & 1
\end{vmatrix}=2\Re\;(A_2^2\overline{A_2})-2|A_2|^2-|A_3|^2+1\]
over the class $\widehat{\mathcal{F}}_{\lambda}(\mathbb{B})$.

\begin{theo}\label{HT1}
Let $g\in H(\mathbb{B},\mathbb{C})$ satisfy $g(0)=1$, and define $F(x)=g(x)x, x\in\mathbb{B}.$
Suppose that $F\in\widehat{\mathcal{F}}_{\lambda}(\mathbb{B})$. Then, for every
$x_{0}\in X$ with $\|x_{0}\|=1$,
\[
-\frac{(2\lambda-1)^2(2\lambda+5)^2}{64\lambda(2\lambda+3)}\leq T_{3,1}(F)\leq  \begin{cases} 1 &\lambda\in [1/2, (\sqrt{153}-5)/8]\\
1+\frac{(1+2\lambda)^2(4\lambda^2+5\lambda-8)}{18}, &\lambda\in ((\sqrt{153}-5)/8,1]
\end{cases}
\]
 The inequalities are sharp.
\end{theo}
\begin{proof} Fix $x_0\in \partial{\mathbb{B}}$ and let $T_{x_0}\in T(x_0)$. Then applying Lemma \ref{L3} we deduce that 

\begin{align*}
T_{3,1}(F)
&=2\Re\;(A_2^2\overline{A_2})-2|A_2|^2-|A_3|^2+1\\
&=2\Re\;(A_2^2\overline{a_2})-2|a_2|^2-|a_3|^2+1\\
&=T_{3,1}(f).
\end{align*}
Hence, from Theorem D (see \cite[Theorem 5.1]{ALT2022}), we obtain that
\[
-\frac{(2\lambda-1)^2(2\lambda+5)^2}{64\lambda(2\lambda+3)}\leq T_{3,1}(F)\leq 
\begin{cases} 1 &\lambda\in [1/2, (\sqrt{153}-5)/8]\\
1+\frac{(1+2\lambda)^2(4\lambda^2+5\lambda-8)}{18}, &\lambda\in ((\sqrt{153}-5)/8,1]
\end{cases}
\]

\end{proof}
The following example shows that the upper bound in Theorem \ref{HT1} is sharp.

\begin{exm} To prove sharpness of upper bound we devide two cases:\\

{\bf Case I:} First suppose that  $\lambda\in [1/2,(\sqrt{153}-5)/8]]$. We consider the mapping
\[
F(x)=
\frac{f\!\left(T_{x_0}(x)\right)}{T_{x_0}(x)}\,x,
\qquad
x_0\in\partial\mathbb{B},\;
T_{x_0}\in T(x_0),
\]
where $f(z)=z$, for $z\in\mathbb{U}.$
By Lemma~\ref{L2} we have $F\in \widehat{\mathcal{F}}_{\lambda}(\mathbb{B})$.

Furthermore, a straightforward computation gives
\beas
\frac{D^2F(0)(x^2)}{2!}=0
\;\;\;\text{and}\;\;\;\frac{D^3F(0)(x^3)}{3!}
=0.
\eeas

Hence $A_2=0$ and $A_3=0$. 
Therefore,
\[
T_{3,1}(F)
=2\Re\;(A_2^2\overline{A_2})-2|A_2|^2-|A_3|^2+1=1.\]

\medskip

{\bf Case II:} Next suppose that  $\lambda\in ((\sqrt{153}-5)/8,1]$. We consider the mapping
\[
F(x)=
\frac{f\!\left(T_{x_0}(x)\right)}{T_{x_0}(x)}\,x,
\qquad
x_0\in\partial\mathbb{B},\;
T_{x_0}\in T(x_0),
\]
where $f(z)=\dfrac{1-(1-z)^{-2\lambda}}{2\lambda}$, for $z\in\mathbb{U}.$
By Lemma~\ref{L2} we have $F\in \widehat{\mathcal{F}}_{\lambda}(\mathbb{B})$.

Furthermore, a straightforward computation gives
\beas
\frac{D^2F(0)(x^2)}{2!}=\frac{1}{2}(1+2\lambda)T_{x_0}(x)x
\;\;\;\text{and}\;\;\;\frac{D^3F(0)(x^3)}{3!}
=\frac{1}{3}(1+2\lambda)(1+\lambda)T_{x_0}(x)^2x.
\eeas

Hence $A_2=\frac{1}{2}(1+2\lambda)$ and $A_3=\frac{1}{3}(1+2\lambda)(1+\lambda)$. 
Therefore,
\[
T_{3,1}(F)
=2\Re\;(A_2^2\overline{A_2})-2|A_2|^2-|A_3|^2+1=1+\frac{(1+2\lambda)^2(4\lambda^2+5\lambda-8)}{18}.\]

Thus, the upper bound obtained in Theorem \ref{HT1} is attained, showing that the estimate is sharp.\\

\end{exm}

The following example shows that the lower bound in Theorem \ref{HT1} is sharp.

\begin{exm}
Consider the mapping
\[
F(x)=
\frac{f\!\left(T_{x_0}(x)\right)}{T_{x_0}(x)}\,x,
\qquad
x_0\in\partial\mathbb{B},\;
T_{x_0}\in T(x_0),
\]
where
\[
f(z)=\int_{0}^z(1-at+t^2)^{-\left(\frac{1}{2}+\lambda\right)}dt, 
\qquad z\in\mathbb{U},
\]
where $a=2\sqrt{\frac{2\lambda+15}{8\lambda(2\lambda+3)}}$. 
By Lemma~\ref{L2} we have $F\in \widehat{\mathcal{F}}_{\lambda}(\mathbb{B})$.

Furthermore, a straightforward computation gives
\beas
\frac{D^2F(0)(x^2)}{2!}&=&\frac{1+2\lambda}{2}\sqrt{\frac{2\lambda+15}{8\lambda(2\lambda+3)}}T_{x_0}(x)x,\\
\;\;\;\text{and}\;\;\;\frac{D^3F(0)(x^3)}{3!}
&=&\frac{(1+2\lambda)(-4\lambda^2+4\lambda+15)}{16\lambda(2\lambda+3)}\bigl(T_{x_0}(x)\bigr)^2x.
\eeas

Hence,
\[
A_2=\frac{1+2\lambda}{2}\sqrt{\frac{2\lambda+15}{8\lambda(2\lambda+3)}},\;\;
\text{and}\;\;A_3=\frac{(1+2\lambda)(-4\lambda^2+4\lambda+15)}{16\lambda(2\lambda+3)}.
\]
Therefore,
\[
T_{3,1}(F)
=2\Re\;(A_2^2\overline{A_2})-2|A_2|^2-|A_3|^2+1
=-\frac{(2\lambda-1)^2(2\lambda+5)^2}{64\lambda(2\lambda+3)}.\]

Thus, the lower bound obtained in Theorem \ref{HT1} is attained, showing that the estimate is sharp.
\end{exm}

Next, removing the restrictive assumption $F(x)=g(x)x$, we generalize Theorem D to higher dimensions under weaker assumptions than those of Theorem \ref{HT1}. 
\begin{theo}\label{HT2}
Let $F$ be a locally biholomorphic mapping on $\mathbb{B}$, and suppose that $F$ satisfies the assumption \eqref{cc1}. If $F\in \widehat{\mathcal{F}}_{\lambda}(\mathbb{B})$, then for every $x_0\in X$ with $\|x_0\|=1$, we have
\[
-\frac{(2\lambda-1)^2(2\lambda+5)^2}{64\lambda(2\lambda+3)}\leq T_{3,1}(F)\leq  \begin{cases} 1 &\lambda\in [1/2, (\sqrt{153}-5)/8]\\
1+\frac{(1+2\lambda)^2(4\lambda^2+5\lambda-8)}{18}, &\lambda\in ((\sqrt{153}-5)/8,1]
\end{cases}
\]
 The inequalities are sharp.
\end{theo}
\begin{proof} Fix $x_0\in \partial{\mathbb{B}}$ and let $T_{x_0}\in T(x_0)$. Next, applying Lemma \ref{L4} and proceeding analogously to the proofs of Theorems D (see \cite[Theorem 5.1]{ALT2022}), we deduce that

\[
-\frac{(2\lambda-1)^2(2\lambda+5)^2}{64\lambda(2\lambda+3)}\leq T_{3,1}(F)\leq  \begin{cases} 1 &\lambda\in [1/2, (\sqrt{153}-5)/8]\\
1+\frac{(1+2\lambda)^2(4\lambda^2+5\lambda-8)}{18}, &\lambda\in ((\sqrt{153}-5)/8,1]
\end{cases}
\]
 The sharpness of the result is an immediate consequence of Theorem \ref{HT1}. This completes the proof.
\end{proof}
\section{{\bf The Zalcman functional $J_{2,3}(F)$ for class $\widehat{\mathcal{F}}_{\lambda}(\mathbb{B})$.}}
In this section, we compute sharp upper bounds for $J_{2,3}(F)|=|A_2A_3-A_4|$
over the class $\widehat{\mathcal{F}}_{\lambda}(\mathbb{B})$.

\begin{theo}\label{J1}
Let $g\in H(\mathbb{B},\mathbb{C})$ satisfy $g(0)=1$, and define $F(x)=g(x)x, x\in\mathbb{B}.$
Suppose that $F\in\widehat{\mathcal{F}}_{\lambda}(\mathbb{B})$. Then, for every
$x_{0}\in X$ with $\|x_{0}\|=1$,
\[
|J_{2,3}(F)|\leq \frac{(1+2\lambda)(7+2\lambda)^{3/2}}{36\sqrt{3(9-4\lambda^2)}}
\]
 The inequality is sharp.
\end{theo}
\begin{proof} Fix $x_0\in \partial{\mathbb{B}}$ and let $T_{x_0}\in T(x_0)$. Then applying Lemma \ref{L3}, we deduce that 

\begin{align*}
|J_{2,3}(F)|
&=|A_2A_3-A_4|\\
&=|a_2a_3-a_4|\\
&=J_{2,3}(f).
\end{align*}
Hence, from Theorem E (see \cite[Theorem 6.1]{ALT2022}), we obtain that
\[
|J_{2,3}(F)|\leq \frac{(1+2\lambda)(7+2\lambda)^{3/2}}{36\sqrt{3(9-4\lambda^2)}}.
\]

\end{proof}

The following example shows that the upper bound in Theorem \ref{J1} is sharp.

\begin{exm}
Consider the mapping
\[
F(x)=
\frac{f\!\left(T_{x_0}(x)\right)}{T_{x_0}(x)}\,x,
\qquad
x_0\in\partial\mathbb{B},\;
T_{x_0}\in T(x_0),
\]
where
\[
f(z)=\int_{0}^z(1-u_0t+t^2)^{-\left(\frac{1}{2}+\lambda\right)}dt, 
\qquad z\in\mathbb{U},
\]
where $u_0=2\sqrt{\frac{2\lambda+7}{3\lambda(9-4\lambda^2)}}$. 
By Lemma~\ref{L2} we have $F\in \widehat{\mathcal{F}}_{\lambda}(\mathbb{B})$.

Furthermore, a straightforward computation gives
\beas
\frac{D^2F(0)(x^2)}{2!}&=&\frac{1+2\lambda}{2}\sqrt{\frac{2\lambda+7}{3(9-4\lambda^2)}}T_{x_0}(x)x,\\
\frac{D^3F(0)(x^3)}{3!}
&=&\frac{(1+2\lambda)\bigl(16\lambda^2+8\lambda-5\bigr)}{18(9-4\lambda^2)}\bigl(T_{x_0}(x)\bigr)^2x\\
\;\;\;\text{and}\;\;\;\frac{D^4F(0)(x^4)}{4!}&=&\frac{(1+2\lambda)\left(40\lambda^3+68\lambda^2-4\lambda-51\right)}{18(9-4\lambda^2)} \sqrt{\frac{2\lambda+7}{3(9-4\lambda^2)}}\bigl(T_{x_0}(x)\bigr)^3x.
\eeas

Hence,
\beas
A_2&=&\frac{1+2\lambda}{2}\sqrt{\frac{2\lambda+7}{3(9-4\lambda^2)}}\\
A_3&=&\frac{(1+2\lambda)\bigl(16\lambda^2+8\lambda-5\bigr)}{18(9-4\lambda^2)}\\
A_4&=&\frac{(1+2\lambda)\left(40\lambda^3+68\lambda^2-4\lambda-51\right)}{18(9-4\lambda^2)}\sqrt{\frac{2\lambda+7}{3(9-4\lambda^2)}}.
\eeas
Therefore,
\[
|J_{2,3}(F)|
=|A_2A_3-A_4|
=\frac{(1+2\lambda)(7+2\lambda)^{3/2}}{36\sqrt{3(9-4\lambda^2)}}.\]

Thus, the bound obtained in Theorem \ref{J1} is attained, showing that the estimate is sharp.
\end{exm}

Next, removing the restrictive assumption $F(x)=g(x)x$, we generalize Theorem E to higher dimensions under weaker assumptions than those of Theorem \ref{J1}. 
\begin{theo}\label{J2}
Let $F$ be a locally biholomorphic mapping on $\mathbb{B}$, and suppose that $F$ satisfies the assumption \eqref{cc1}. If $F\in \widehat{\mathcal{F}}_{\lambda}(\mathbb{B})$, then for every $x_0\in X$ with $\|x_0\|=1$, we have
\[
|J_{2,3}(F)|\leq \frac{(1+2\lambda)(7+2\lambda)^{3/2}}{36\sqrt{3(9-4\lambda^2)}}.
\] 
 The inequality is sharp.
\end{theo}
\begin{proof} Fix $x_0\in \partial{\mathbb{B}}$ and let $T_{x_0}\in T(x_0)$. Next, applying Lemma \ref{L4}, we deduce that 
\[J_{2,3}(F)=\frac{(1+2\lambda)^3}{192}p_1^3-\frac{(1+2\lambda)^2}{36}p_1p_2-\frac{1+2\lambda}{12}p_3.\]
Proceeding analogously to the proofs of Theorems E (see \cite[Theorem 6.1]{ALT2022}), we deduce that

\[
|J_{2,3}(F)|\leq \frac{(1+2\lambda)(7+2\lambda)^{3/2}}{36\sqrt{3(9-4\lambda^2)}}.
\]
 The sharpness of the result is an immediate consequence of Theorem \ref{J1}. This completes the proof.
\end{proof}

\section*{{\bf Declarations}}
\subsection*{Funding} Not applicable.
\subsection*{Data Availability Statement}
Data sharing is not applicable to this article as no datasets were generated or analyzed during the current study.
\subsection*{Conflict of Interest}
The authors declare that they have no conflict of interest. 
\subsection*{Author Contributions}
Both authors contributed equally to this work.

\end{document}